\documentclass[12pt, reqno]{amsart}
\usepackage{ amsmath,amsthm, amscd, amsfonts, amssymb, graphicx, color}
\usepackage[bookmarksnumbered, colorlinks, plainpages,linkcolor=blue,urlcolor=blue,citecolor=blue]{hyperref}
\usepackage{setspace}
\usepackage{ulem}

\newtheorem{thm}{Theorem}[section]
\newtheorem{cor}[thm]{Corollary}
\newtheorem{lem}[thm]{Lemma}

\newtheorem{exam}[thm]{Example}
\numberwithin{equation}{section}

\begin{document}

\title{Rings generated by idempotents and nilpotents}

\author{Huanyin Chen}
\author{Marjan Sheibani}
\address{
School of Mathematics\\ Hangzhou Normal University\\ Hang -zhou, China}
\email{<huanyinchenhz@163.com>}
\address{Farzanegan Campus, Semnan University, Semnan, Iran}
\email{<m.sheibani@semnan.ac.ir>}

\subjclass[2010]{16U99, 16E50.} \keywords{idempotent; tripotent; nilpotent; strongly 2-nil-clean ring; Zhou nil-clean ring.}

\begin{abstract} We present new characterizations of the rings in which every element is the sum of two idempotents and a nilpotent that commute,
and the rings in which every element is the sum of two tripotents and a nilpotent that commute. We prove that such rings are completely determined by the additive decompositions of their square elements. These improve the results of Chen and Sheibani[J. Algebra Appl., 16, 1750178(2017)] and Zhou [J. Algebra Appl., 16, 1850009(2017)].
\end{abstract}

\maketitle

\section{Introduction}

A ring $R$ is strongly 2-nil-clean ring if every element in $R$ is the sum of two idempotents and a nilpotent that commute. An element $p\in R$ is tripotent if $p^3=p$. A ring $R$ is Zhou nil-clean ring if every element in $R$ is the sum of two tripotents and a nilpotent that commute. Many elementary properties and structure theorems of such rings were investigated in~\cite{A1,CS,CS2,CS3,K2,Y,Z}. In this paper we shall characterize preceding rings by means of the additive decomposition of their
square elements. These improve the known results, e.g., ~\cite[Theorem 16]{A}, ~\cite[Theorem 2.8]{CS} and ~\cite[Theorem 2.11]{Z}.

In Section 2, we prove that a ring $R$ is strongly 2-nil-clean if and only every square element in $R$ is the sum of two idempotents and a nilpotent that commute if and only if every square element in $R$ is the sum of an idempotent, an involution and a nilpotent that commute.

In Section 3, we further prove that a ring $R$ is Zhou nil-clean if and only if every square element in $R$ is the sum of two tripotents and a nilpotent that commute if and only if every square element in $R$ is the sum of a tripotent, an involution and a nilpotent that commute.

Throughout, all rings are associative with an identity. We use $N(R)$ to denote the set of all nilpotents in $R$. $a\in R$ is square if $a=x^2$ for some $x\in R$. $a\equiv b (mod~N(R))$ means that $a-b\in N(R)$.

\section{strongly 2-nil-clean rings}

In this section, we establish new characterizations of strongly 2-nil-clean rings by means of their square elements. We begin with

\begin{lem} The following are equivalent for a ring $R$:\end{lem}
\begin{enumerate}
\item [(1)] {\it $R$ is strongly 2-nil-clean.}
\vspace{-.5mm}
\item [(2)] {\it For any $a\in R$, $a-a^3\in N(R)$.}
\vspace{-.5mm}
\item [(3)] {\it Every element in $R$ is the sum of a tripotent and a nilpotent that commute.}
\end{enumerate}
\begin{proof} See ~\cite[Theorem 2.3 and Theorem 2.8]{CS}.\end{proof}

\begin{thm} The following are equivalent for a ring $R$:\end{thm}
\begin{enumerate}
\item [(1)] {\it $R$ is strongly 2-nil-clean.}
\vspace{-.5mm}
\item [(2)] {\it Every square element in $R$ is the sum of an idempotent and a nilpotent that commute.}
\vspace{-.5mm}
\item [(3)] {\it Every square element in $R$ is the sum of two idempotents and a nilpotent that commute.}
\vspace{-.5mm}
\item [(4)] {\it Every square element in $R$ is the sum of three idempotents and a nilpotent that commute.}
\end{enumerate}
\begin{proof} $(1)\Rightarrow (2)$ For any $a\in R$, it follows by Lemma 2.1 that $a-a^3\in N(R)$. Hence $a^2-(a^2)^2=a(a-a^3)\in N(R)$. In view of~\cite[Lemma 3.5]{Y}, $a^2$ is the sum of an idempotent and a nilpotent that commute.

$(2)\Rightarrow (3)\Rightarrow (4)$ These are trivial.

$(4)\Rightarrow (1)$ Step 1. By hypothesis, there exist idempotents $e,f,g\in R$ and a nilpotent $w\in R$ that commute such that
$2^2=e+f+g+w$. Hence $4e=e+ef+eg+ew$, and so $3e=ef+eg+ew$. This implies that $3ef=ef+efg+efw$; whence,
$2ef=efg+efw$. Therefore $$2ef=(2ef)^2-2ef=(efg+efw)^2-(efg+efw)\in N(R).$$ Likewise, we have $2eg, 2fg\in N(R)$.
Accordingly, $$\begin{array}{lll}
12&=&4^2-4\\
&=&(e+f+g+w)^2-(e+f+g+w)\\
&\equiv &2(ef+eg+fg) (mod~N(R)),
\end{array}$$ i.e., $6^2=3\times 12\in N(R)$. Hence $6\in N(R)$.
Write $2^n\times 3^n=0$ for some $n\in {\Bbb N}$. Since $2^nR\bigcap 3^nR=0$, we have $R\cong R_1\times R_2$, where $R_1=R/2^nR, R_2=R/3^nR$.

Step 2. Let $a\in R$. Then there exist idempotents $e,f,g\in R$ and a nilpotent $w\in R$ that commute such that $(a+1)^2=e+f+g+v$, and so
$a^2+2a=e-(1-f)+g+w=e+h-k+w$, where $h=fg, k=(1-f)(1-g)$. We easily check that $h^2=h,k^2=k$ and $hk=kh=0$.
Hence, we have
$$\begin{array}{rll}
(a^2+2a)^3&\equiv &(e+h+k+2eh-2ek)(e+h-k)\\
&\equiv &e+h-k+6eh\\
&\equiv &a^2+2a (mod~N(R)).
\end{array}$$
Therefore $(a^2+2a)^3-(a^2+2a)\in N(R).$ Likewise, $(a-1)^2$ is the sum of three idempotents and a nilpotent that commute.
As the preceding discussion, we have $(a^2-2a)^3-(a^2-2a)\in N(R).$

Case 1. $a\in R_1$. Then $2\in N(R_1)$. Hence, $(a^3-a)(a^3+a)=a^6-a^2\in N(R_1)$. This implies that
$(a^3-a)^2\in N(R_1)$, i.e., $a-a^3\in N(R_1)$.

Case 2. $a\in R_2$. Since $3\in N(R_2)$, we get  $$\begin{array}{c}
(a^2-a)^3-(a^2-a)\in N(R_2),\\
(a^2+a)^3-(a^2+a)\in N(R_2).
\end{array}$$
Moreover, we have $$\begin{array}{c}
(a^6-a^3)-(a^2-a)\in N(R_2),\\
(a^6+a^3)-(a^2+a)\in N(R_2).
\end{array}$$ Thus, $2a^3-2a\in N(R_2)$. Clearly, $2\in R_2^{-1}$, and then $a-a^3\in N(R_2)$.

Therefore for any $a\in R$, we have $a-a^3\in N(R)$. In light of Lemma 2.1, $R_2$ is strongly 2-nil-clean.\end{proof}

We note that "three idempotents" in the proceeding theorem can not be replaced by "four idempotents" as the following shows.

\begin{exam} Let $R={\Bbb Z}_5$. Then every square element in $R$ is the sum of four idempotents and a nilpotent that commute.
But $R$ it is not strongly 2-nil-clean.
\end{exam}\begin{proof} Obviously, $\{ x^2\mid x\in R\}=\{ 0,1,4\}$. For any $a\in R$, we see that $a^2$ is the sum of four idempotents that commute.
As $2\neq 2^3$, $R$ is not strongly 2-nil-clean.\end{proof}

An element $v\in R$ is an involution if $v^2=1$. As a consequence of Theorem 2.2, we now derive

\begin{thm} The following are equivalent for a ring $R$:\end{thm}
\begin{enumerate}
\item [(1)] {\it $R$ is strongly 2-nil-clean.}
\vspace{-.5mm}
\item [(2)] {\it Every square element in $R$ is the sum of an idempotent, an involution and a nilpotent that commute.}
\end{enumerate}
\begin{proof} $(1)\Rightarrow (2)$ Let $a\in R$. In view of Theorem 2.2, there exist an idempotent $e\in R$ and a nilpotent $w\in R$ that commute such that $a^2=e+w$.
Hence $a^2=(1-e)+(2e-1)+w$ with $(1-e)^2=1$ and $(2e-1)^2=1$. That is, $a^2$ is the sum of an idempotent, an involution and a nilpotent that commute.

$(2)\Rightarrow (1)$ Write $2^2=e+v+w$, where $e^2=e,v^2=1$ and $w\in N(R)$ that commute.
Then $$16\equiv e+2ev+1\equiv (4-v)+2(4-v)v+1\equiv 3+5v (mod~N(R)).$$ This implies that $13\equiv 5v (mod~N(R))$.
Hence $169\equiv 25v^2=25 (mod~N(R))$, and then $2^4\times 3^2\in N(R)$. That is, $2\times 3=6\in N(R)$.
Write $2^n\times 3^n=0$ for some $n\in {\Bbb N}$. Then we have
$R\cong R_1\times R_2$, where $R_1=R/2^nR$ and $R_2=R/3^nR$.

Step 1. Let $a\in R_1$. Then there exist $e,v,w\in R_1$ such that $a^2=e+v+w$, $e^2=e,v^2=1$ and $w\in N(R_1)$ that commute. Clearly, $2\in N(R_1)$. Then we have
$$\begin{array}{lll}
a^8&\equiv &(e+v)^4\equiv e+1 (~\mbox{mod}~N(R_1)),\\
a^{12}&=&a^8a^4\equiv (e+1)(e+1)\equiv e+1 (~\mbox{mod}~N(R_1)).
\end{array}$$ Hence $a^8(1-a^2)(1+a^2)=a^8(1-a^4)=a^8-a^{12}\in N(R_1)$,
and so $a^6(a-a^3)^2=a^8(1-a^2)^2\in N(R_1)$. This implies that $(a-a^3)^8=a^6(a-a^3)^2(1-a^2)^6\in N(R_1)$, i.e., $a-a^3\in N(R_1)$. According to Lemma 2.1,
$R_1$ is strongly 2-nil-clean.

Step 2. Let $a\in R_2$. Then there exist $e,v,w\in R_2$ such that $a^2=e+v+w$, $e^2=e,v^2=1$ and $w\in N(R_2)$ that commute. Since $3\in N(R_2)$, we have
$a^6\equiv (e+v)^3\equiv e+v \equiv a^2 (mod~N(R_2)).$ Hence $a^2-a^6\in N(R_2)$.
Clearly, $2\in R_2^{-1}$. Set $p=\frac{a^4+a^2}{2},q=\frac{a^4-a^2}{2}$. We compute that $$\begin{array}{c}
p^2-p=\frac{1}{4}(a^8+2a^6-a^4-2a^2)=\frac{1}{4}(a^2+2)(a^6-a^2),\\
q^2-q=\frac{1}{4}(a^8-2a^6-a^4+2a^2)=\frac{1}{4}(a^2-2)(a^6-a^2).
\end{array}$$ Hence $p^2-p,q^2-q\in N(R_2)$. In light of~\cite[Lemma 3.5]{Y},
we can find two idempotents $g,h\in {\Bbb Z}[a]$ such that $p-g,q-h\in N(R_2)$.
Hence $a^2=g-h+w$ for some $w\in {\Bbb Z}[a]$. As $3\in N(R_2)$, we see that
$a^2=g+h+h+(w-3h)$ with $w-3h\in N(R_2)$. Therefore
$a^2$ is the sum of three idempotents and a nilpotent that commute.
According to Theorem 2.2, $R_2$ is strongly 2-nil-clean.

Therefore $R\cong R_1\times R_2$ is strongly 2-nil-clean, as asserted.\end{proof}

\section{Zhou nil-clean rings}

The aim of this section is to further characterize Zhou nil-clean rings by means of the additive decompositions of their square elements.  An element $p\in R$ is 5-potent if $p^5=p$. We have

\begin{lem} The following are equivalent are equivalent for a ring $R$:\end{lem}
\begin{enumerate}
\item [(1)] {\it $R$ is Zhou nil-clean.}
\vspace{-.5mm}
\item [(2)] {\it For any $a\in R$, $a-a^5\in N(R)$.}
\vspace{-.5mm}
\item [(3)] {\it Every element in $R$ is the sum of a 5-potent and a nilpotent that commute.}
\end{enumerate}
\begin{proof} See ~\cite[Theorem 19]{A} and ~\cite[Theorem 2.11]{Z}.\end{proof}

\begin{lem} A ring $R$ is Zhou nil-clean if and only if $7\in R^{-1}$ and every square element in $R$ is the sum of four idempotents and a nilpotent that commute.\end{lem}
\begin{proof} $\Longrightarrow $ In view of Lemma 3.1, $30=2^5-2\in N(R)$. Since $(7, 30)=1$, we can find some $k,l\in {\Bbb N}$ such that
$7k+30l=1$, and so $7\in R^{-1}$. By virtue of~\cite[Theorem 2.11]{Z}, $R\cong A\times B\times C$, where $A=0$ or $A/J(A)$ is Boolean with $J(A)$ nil, $B=0$ or $B/J(B)$ is a subdirect product of ${\Bbb Z}_3$'s with $J(B)$ nil; $C=0$ or $C/J(C)$ is a subdirect product of ${\Bbb Z}_5$'s with $J(C)$ nil.
In light of~\cite[Theorem 15]{A}, every element in $R$ is the sum of four idempotents and a nilpotent that commute.

$\Longleftarrow $ Step 1. By hypothesis, there exists idempotents $e,f,g,h\in R$ and a nilpotent $w\in R$ that commute such that
$3^2=e+f+g+h+w$. Hence $7=\alpha +\beta +w$, where $\alpha =e-(1-f), \beta=g-(1-h).$
Since $e,1-f,g,1-h$ are idempotents, we easily check that $\alpha^3=\alpha, \beta^3=\beta$ and $\alpha\beta=\beta\alpha$. Then $48\times 7=7^3-7\equiv 3\alpha\beta (\alpha+\beta) (mod~N(R)),$ hence, $48(\alpha+\beta)\equiv 3\alpha\beta (\alpha+\beta) (mod~N(R))$. Multiplying both sides by $\alpha\beta$ gives
$$48\alpha\beta(\alpha+\beta)\equiv 3\alpha\beta (\alpha+\beta)\alpha\beta\equiv 3\alpha\beta (\alpha+\beta) (mod~N(R)),$$ and so
$$\begin{array}{lll}
48\times 7\times 15&\equiv&15[3\alpha\beta (\alpha+\beta)]\\
&=&45\alpha\beta(\alpha+\beta)\\
&\equiv& 0 (mod~N(R)).
\end{array}$$ Hence, $2^4\times 3^2\times 5\times 7\in N(R)$. It follows from $7\in R^{-1}$ that
$2^4\times 3^2\times 5\in N(R)$, and so $2\times 3\times 5\in N(R)$. Write $2^n\times 3^n\times 5^n=0$ for some $n\in {\Bbb N}$.
Then $R\cong R_1\times R_2\times R_3$, where $R_1=R/2^nR,R_2=R/3^nR,R_3=R/5^nR$.

Step 2. By hypothesis, there exist idempotents $e,f,g,h\in R$ and a nilpotent $w\in R$ that commute such that $(a+2)^2=e+f+g+h+w$, and so
$a^2+4a+2=e-(1-f)+g-(1-h)+w$. Set $p=e-(1-f)$ and $q=g-(1-h)$. Then $a^2+4a+2=p+q (mod~N(R)), p^3=p, q^3=q$ and $pq=qp$.

Case 1. $a\in R_1$. Then $2\in N(R_1)$. We have
$$\begin{array}{rll}
a^8&\equiv&(p+q)^4\\
&\equiv&p^4+q^4 \\
&\equiv&p^2+q^2\\
&\equiv&a^4 (mod~N(R_1)),
\end{array}$$
Hence $a^3(a-a^5)\in N(R_1)$, and so $(a-a^5)^4=a^3(a-a^5)(1-a^4)^3\in N(R_1).$ Therefore $a-a^5\in N(R_1)$.

Case 2. $a\in R_2$. Then $3\in N(R_2)$. We have
$$\begin{array}{rll}
(a^2+4a+2)^3&\equiv&(p+q)^3 \\
&\equiv&p^3+q^3\\
&\equiv&p+q \\
&\equiv&a^2+4a+2 (mod~N(R_2)),
\end{array}$$
Likewise, we have $(a^2-4a+2)^3-(a^2-4a+2)\in N(R_2).$ So we get $$\begin{array}{c}
(a^2+2)^3+(4a)^3-(a^2+4a+2)\in N(R_2),\\
(a^2+2)^3-(4a)^3-(a^2-4a+2)\in N(R_2).
\end{array}$$
Hence $2(4a)^3-8a\in N(R_2)$, and so $2^3(16a^3-a)\equiv 2^3(a^3-a)\in N(R_2)$. Accordingly, $a^3-a\in N(R_2)$, and so $a-a^5=(1+a^2)(a-a^3)\in N(R_2)$.

Case 3. $a\in R_3$. Then $5\in N(R_3)$. We have
$$\begin{array}{rll}
(a^2+4a+2)^5&\equiv&(p+q)^5\\
&\equiv&p^5+q^5\\
&\equiv&p+q\\
&\equiv&a^2+4a+2 (mod~N(R_3)),
\end{array}$$
Likewise, we get $(a^2-4a+2)^5-(a^2-4a+2)\in N(R_3)$.
Then $$\begin{array}{c}
(a^2+2)^5+(4a)^5-(a^2+4a+2)\in N(R_3),\\
(a^2+2)^5-(4a)^5-(a^2-4a+2)\in N(R_3).
\end{array}$$
Hence $2(4a)^5-8a\in N(R_3)$, and so $2^3(256a^5-a)=2^3(a^5-a)\in N(R)$. Accordingly, $a-a^5\in N(R_3)$.

Therefore $a-a^5\in N(R)$ for all $a\in R$. This completes the proof by Lemma 3.1.\end{proof}

We are now ready to prove the following.

\begin{thm} The following are equivalent for a ring $R$:\end{thm}
\begin{enumerate}
\item [(1)] {\it $R$ is Zhou nil-clean.}
\vspace{-.5mm}
\item [(2)] {\it Every square element in $R$ is the sum of a tripotent and a nilpotent that commute.}
\vspace{-.5mm}
\item [(3)] {\it Every square element in $R$ is the sum of two tripotents and a nilpotent that commute.}
\end{enumerate}
\begin{proof} $(1)\Rightarrow (2)$ In view of~\cite[Theorem 2.11]{Z}, $R\cong R_1\times R_2$, where $R_1$ is strongly 2-nil-clean with $2\in N(R_1)$ and $R_2$ is Zhou nil-clean with $3\times 5\in N(R_2)$.

Let $a\in R_1$. Then $a^2$ is the sum of an idempotent and a nilpotent that commute.

Let $a\in R_2$. In view of ~\cite[Theorem 2.11]{Z}, $a-a^5\in N(R_2)$. Hence,
$a^2-(a^2)^3=a(a-a^5)\in N(R_2)$. Since $2\in R_2^{-1}$, it follows by ~\cite[Lemma 2.6]{Z} that there exists a tripotent $p\in {\Bbb Z}[a]$ such that
$a^2-p\in N(R_2)$.

Therefore every square element in $R$ is the sum of a tripotent and a nilpotent that commute, as required.

$(2)\Rightarrow (3)$ This is trivial.

$(3)\Rightarrow (1)$ By hypothesis, there exist tripotents $e,f\in R$ a nilpotent $w\in R$ that commute such that $2^2=e+f+w$.
We check that $$15(e+f)\equiv 4(4^2-1)=4^3-4\equiv 3ef(e+f) (mod~N(R)).$$ Multiplying both sides by $ef$ gives
$$15ef(e+f)\equiv 3ef(e+f) (mod~N(R)).$$ This implies that $$2^4\times 3\times 5=4\times (4^3-4)\equiv 12ef(e+f)\equiv 0 (mod~N(R)).$$ It follows that
$2\times 3\times 5\in N(R)$. Write $2^n\times 3^n\times 5^n=0$ for some $n\in {\Bbb N}$. Then $$R\cong R_1\times R_2\times R_3, R_1=R/2^nR, R_2=R/3^nR, R_3=R/5^nR.$$

Case 1. $a\in R_1$. Then $2\in N(R_1)$. By hypothesis, there exist tripotents $p,q\in R_1$ and a nilpotent $w\in R_1$ that commute such that $a^2=p+q+w$. Hence we have
$$\begin{array}{rll}
(a^2)^4&\equiv&(p+q)^4\\
&\equiv&p^4+q^4 \\
&\equiv&p^2+q^2\\
&\equiv&a^4 (mod~N(R_2)),
\end{array}$$
Hence $a^3(a-a^5)\in N(R_1)$; whence, $(a-a^5)^4=a^3(a-a^5)(1-a^4)^3\in N(R_1).$ This implies that $a-a^5\in N(R_1)$.

Case 2. $a\in R_2$. Then $3\in N(R_2)$. By hypothesis, there exist tripotents $p,q\in R_2$ and a nilpotent $w\in R_2$ that commute such that $a^2=p+q+w$. Hence we have
$$\begin{array}{rll}
(a^2)^3&\equiv&(p+q)^3 \\
&\equiv&p^3+q^3\\
&\equiv&p+q \\
&\equiv&a^2 (mod~N(R_2)),
\end{array}$$
This implies that $a(a-a^5)\in N(R_2)$, and so $(a-a^5)^2=a(a-a^5)(1-a^4)\in N(R_2)$. Accordingly, $a-a^5\in N(R_2)$.

Case 3. $a\in R_3$. Then $5\in N(R_3)$. By hypothesis, there exist tripotents $p,q\in R_3$ and a nilpotent $w\in R_3$ that commute such that $(a+2)^2=p+q+w$. Hence we get
$$\begin{array}{rll}
(a^2+4a+4)^5&\equiv&(p+q)^5\\
&\equiv&p^5+q^5\\
&\equiv&p+q\\
&\equiv&a^2+4a+4 (mod~N(R_3)),
\end{array}$$
Likewise, we have $(a^2-4a+4)^5\equiv a^2-4a+4 (mod~N(R_3))$.
We derive that $$\begin{array}{c}
(a^2-a-1)^5-(a^2-a-1)\in N(R_3),\\
(a^2+a-1)^5-(a^2+a-1)\in N(R_3).
\end{array}$$
This implies that $$\begin{array}{c}
(a^2-1)^5-a^5-(a^2-1-a)\in N(R_3),\\
(a^2-1)^5+a^5-(a^2-1+a)\in N(R_3).
\end{array}$$
Thus we have $2a^5-2a\in N(R_3)$. Clearly, $2\in R_3^{-1}$, and therefore $a-a^5\in N(R_3)$.

According to Lemma 3.1, $R\cong R_1\times R_2\times R_3$ is Zhou nil-clean.\end{proof}

As a consequence, we now derive:

\begin{cor} The following are equivalent for a ring $R$:\end{cor}
\begin{enumerate}
\item [(1)] {\it $R$ is Zhou nil-clean.}
\vspace{-.5mm}
\item [(2)] {\it Every square element in $R$ is the sum of a tripotent, an involution and a nilpotent that commute.}
\end{enumerate}
\begin{proof}  $(1)\Rightarrow (2)$ Clearly, $2\times 15=30\in N(R)$. Then we can find some $n\in {\Bbb N}$ such that
$$R\cong R_1\times R_2, R_1=R/2^nR, R_2=R/15^nR.$$

Case 1. Let $a\in R_1$. Clearly, $R_1$ is strongly 2-nil-clean. By virtue of Theorem 2.4, $a^2$ is the sum of an idempotent, an involution $u$ and a nilpotent that commute.

Case 2. Let $a\in R_2$. In view of Theorem 3.3, there exist a tripotent $p\in R$ and a nilpotent $w\in R$ that commute such that
$a^2=p+w$. Clearly, $2\in R_2^{-1}$. Let $e=\frac{p^2+p}{2}$ and $f=\frac{p^2-p}{2}$. Then we directly verify that
$$e^2-e=\frac{1}{4}(p^4+2p^3-p^2-2p)=\frac{1}{4}(p+2)(p^3-p).$$
Hence $e^2=e$. Likewise, $f^2=f$. Therefore
$a^2=e-f+w=(1-e)-f+(2e-1)+w$. We check that $[(1-e)-f]^3=(1-e)-f$ and $(2e-1)^2=1$, as desired.

$(2)\Rightarrow (1)$ Since every involution is a tripotent, we complete the proof by Theorem 3.3.\end{proof}

\begin{exam} Let $R=\{ a,b,c,d\}$ be the ring defined by the following operations:
$$ \begin{array}{cc}
\begin{tabular}{c|cccc}
 + & a &b & c & d  \\
  \hline
 a & a &b & c & d   \\
    b & b &a & d & c  \\
   c & c &d & a & b   \\
   d & d &c & b & a  \\
    \end{tabular}&~~~~~~~~~~
    \begin{tabular}{c|cccc}
 $\times$ & a &b & c & d  \\
  \hline
 a & a &a & a & a   \\
    b & a &b & c & d  \\
   c & a &c & d & b   \\
   d & a &d & b &c  \\
    \end{tabular}
    \end{array}$$
Then $x^3$ is an idempotent for all $x\in R$, but $R$ is not Zhou nil-clean.\end{exam}\begin{proof} For all $x\in R$, we check that
$x=x^4$, and so $(x^3)^2=x^3$. That is, $x^3\in R$ is an idempotent. Since $c-c^5=b$ is not nilpotent, $R$ is not Zhou nil-clean.\end{proof}

\begin{exam} Let $R=\{ \left(
\begin{array}{cc}
x&y\\
y&x+y
\end{array}
\right)~\mid~x,y\in {\Bbb Z_3}\}$. Then $R$ is a finite field with $9$ elements.
Let $a\in R$. Then $a=a^9$, and so $a^2=(a^2)^5$, i.e., $a^2$ is 5-potent. Therefore every square element in $R$ can be written as the sum of a 5-potent and a nilpotent that commute. Choose $c=\left(
\begin{array}{cc}
0&1\\
1&1
\end{array}
\right)$. Then $c-c^5=\left(
\begin{array}{cc}
0&-1\\
-1&-1
\end{array}
\right)=\left(
\begin{array}{cc}
1&-1\\
-1&0
\end{array}
\right)^{-1}.$ Hence $c-c^5$ is not nilpotent. According to~\cite[Theorem 2.11]{Z}, $R$ is not Zhou nil-clean.\end{exam}

\end{document}